\documentclass{amsart}
\usepackage{latexsym, bbm, enumerate, amssymb, amsmath}
\usepackage[dvips]{graphicx}

\usepackage{natbib}

\numberwithin{equation}{section}
\theoremstyle{plain}
\newtheorem{theorem}{Theorem}
\theoremstyle{definition}
\newtheorem{lemma}{Lemma}
\newtheorem{definition}{Definition}
\newtheorem{corollary}{Corollary}
\newtheorem{remark}{Remark}

\newcommand{\bigsum}{\displaystyle \sum_{i=1}^n}

\newcommand{\less}{\mathbb{I}_{X_i\leq t}}
\newcommand{\more}{\mathbb{I}_{X_i> t}}
\newcommand{\lessone}{\mathbb{I}_{X\leq t}}
\newcommand{\moreone}{\mathbb{I}_{X> t}}

\newcommand{\ind}{\indent \indent}
\newcommand{\tin}{ t \in (x_-,x_+)}

\begin{document}

\title{On a Clustering Criterion for Dependent Observations}
\author[Bharath]{Karthik Bharath}
\thanks{Department of Statistics,The Ohio State University, 1958 Neil Ave, Columbus, OH-43210. 
{\it E-mail.}
bharath.4@osu.edu}

\begin{abstract}
A univariate clustering criterion for stationary processes satisfying a $\beta$-mixing condition is proposed extending the work of \cite{KB2} to the dependent setup. The approach is characterized by an alternative sample criterion function based on truncated partial sums which renders the framework amenable to various interesting extensions for which limit results for partial sums are available. Techniques from empirical process theory for mixing sequences play a vital role in the arguments employed in the proofs of the limit theorems. 
\end{abstract}


\maketitle
\section{Introduction}
The purpose of this article is two-fold: we propose a clustering criterion for observations from a smooth, invertible distribution function which is noticeably simpler than the one recently proposed by \cite{KB2}; we then demonstrate the power of its simplicity by proving limit theorems for clustering constructs under a dependence setup extending their work from the i.i.d. case. A pleasant by-product of the proposed framework is that the proofs of the limit results, in some cases, follow along similar lines to the ones in \cite{KB2}.

 \cite{JH2} considered the statistical underpinnings of the $k$-means clustering framework and derived asymptotic distributions of a suitably defined criterion function and its maximum. Given a set of observations (data), he defined a cluster to be the subset of the observations which are grouped together on the basis of a point which splits the data by maximizing the between-group sums of squares; in other words, he considered a criterion function which was based on maximizing the between cluster sums of squares as opposed to minimizing the within cluster sums of squares. \cite{KB2} too resorted to maximizing the between cluster sums of squares but deviated from Hartigan's framework and instead considered the zero of a certain function of the derivative of Hartigan's criterion function. Contrary to Hartigan's setup, which required the existence of a fourth moment for asymptotic results, they proved results for their clustering constructs under a second moment assumption with an added smoothness condition on the criterion function. Both works considered i.i.d. data and primarily the case $k=2$, viz. two clusters.   

We extend the approach used in \cite{KB2} to the case of dependent data satisfying a $\beta$-mixing condition and reprove all results under a dependent setup using techniques from empirical processes. We achieve this by employing an alternative sample version of the criterion function used in their paper which, by virtue of its definition, renders itself amenable to a variety of interesting scenarios. The circumscription of the proposed framework to the $\beta$-mixing case is not necessary and is artificial; our intention is to highlight the versatility of our definition of the sample criterion function and its deployment in investigating asymptotic behavior of the $k$-means clustering criterion for dependent observations. The results presented in this article ought to be viewed as a demonstration of the framework's amenability to sequences having a variety of dependence structures for which limit results for partial sums exist. The results presented in this article are also valid to sequences which satisfy a $\phi$-mixing condition since it represents a stronger condition in the sense that a $\phi$-mixing sequence is also $\beta$-mixing. 

Interest in $\beta$-mixing sequences has increased in recent years with attempts to develop stability bounds for learning algorithms. In many machine learning applications the i.i.d. assumption is not valid since the observations obtained from the learning algorithms usually display inherent temporal dependence. In fact, \cite{vidya} argues that $\beta$-mixing is ``just the right"assumption for the analysis of weakly-dependent sample points in machine learning, in particular since several learning results then carry 
over to the non-i.i.d. case. Another important application of $\beta$-mixing sequences is in modeling scenarios involving aperiodic and positively recurrent Markov chains; i.e. if $\{Y_i\}_{i \geq 1}$ is an aperiodic positive recurrent Markov chain, then it satisfies the $\beta$ mixing condition (see p.139 \cite{Rio}). This fact has been employed in several econometric applications; for an overview see \cite{Dou2}. In similar vein, the condition of $\phi$-mixing is also a fairly common assumption while analyzing stability of learning algorithms such as Support Vector Regression (SVR) (\cite{vapnik}), Kernel Ridge Regression (KRR) (\cite{saunders}) and Support Vector Machines (SVM) (\cite{CV}). Since $\phi$-mixing represents a stronger condition than $\beta$-mixing, stability bounds for $\beta$-mixing usually lead to good bounds for $\phi$-mixing too.  In a recent article, \cite{MR} proved stability bounds for stationary $\phi$-mixing and $\beta$-mixing processes with applications in the above mentioned algorithms. It is therefore of considerable interest, at least in machine learning applications, to move over from the i.i.d. setup to the $\beta$-mixing or $\phi$-mixing cases when weak dependence which decays over time is a reasonable modeling assumption. The results in this article can conceivably be employed while analyzing clustering properties of such sequences. 

In section \ref{prelim} we review some of the pertinent constructs from \cite{KB2} and \cite{JH2} and define the quantities of interest in this article: the theoretical and sample criterion functions and their respective zeroes. In section \ref{assumptions} we state our assumptions regarding the conditions on the distribution function $F$ and rate of decay of the $\beta$-mixing coefficients for the rest of the paper. Then, in section \ref{remarks}, we examine the nature of the criterion function as a statistical functional which induces a functional on the space of  c\'{a}dl\'{a}g functions and note some of the difficulties involved in directly using existing results. In section \ref{prep}, in anticipation of the arguments employed in the main results, we prove a few preparatory Lemmas and state two results from the literature which play an important role in our proofs. Section \ref{asymptotics} comprises the main results of this article. We first examine the sample criterion function and prove a weak convergence result; this then provides us with several useful corollaries which will then assist us in examining the sample split point. However, since the ``zero" of our sample criterion function is similar to the one in \cite{KB2}, the proofs of statements regarding its asymptotic behavior follow along almost identical lines as the ones for the empirical split point in \cite{KB2}; for ease of exposition and interests of brevity, we direct the interested reader to their paper for the details of the proofs. Finally, in section \ref{sim}, we provide numerical verification of the limit theorems via a simulation exercise.
\section{Preliminaries}\label{prelim}
Let us first consider the criterion function that was introduced in \cite{JH2} for partitioning a sample into two groups.
For a distribution function $F$,  if $Q$ is the quantile function associated with $F$, then the \emph{split function} of $Q$ at $p\in (0,1)$ is defined as
\begin{equation}\label{split*function}
 B(Q,p) = p\left[\frac{1}{p}\int_{q< p}Q(q)dq\right]^{2}+(1-p)\left[\frac{1}{1-p}\int_{q> p}Q(q)dq\right]^{2}-\left[\int_0^1Q(q)dq\right]^2.
\end{equation}
 A value $p_0 \in (0,1)$ which maximizes the split function is called the \emph{split point}. When $F$ is invertible ($Q$ is the unique inverse), \cite{KB2} considered a different criterion function defined as, for $0<p<1$, 
\[
G(p)=\frac{1}{p}\displaystyle \int_0^p Q(q)dq+\frac{1}{1-p}\displaystyle \int_p^1 Q(q)dq- 2Q(p),
\]
which is a function of the derivative of $B(Q,p)$ with respect to $p$. The zero of $G$ coincides with the maximum of the parabolic split function $B$. The linear statistical functional (of F) $G$ proffers a simple way to determine the split point when $F$ is invertible as opposed to Hartigan's parabolic split function which accommodated a general $F$.

If $X_{(i)},i=1,\ldots,n$ are order statistics corresponding to i.i.d. real-valued observations $X_i$ obtained from $F$, then the empirical counterpart of $G$, christened, the {\it Empirical Cross-over Function} (ECF), was defined in \cite{KB2} as 
\begin{equation*}\label{G_function}
 G_n(p)=\frac{1}{k}\sum_{j=1}^k X_{(j)}-X_{(k)} +\frac{1}{n-k}\sum_{j=k+1}^n X_{(j)}-X_{(k+1)},
\end{equation*}
for $\frac{k-1}{n}\leq p <\frac{k}{n}$
{ and
\begin{equation*}\label{G_function*at*1}
 G_n(p)=\frac{1}{n}\sum_{j=1}^n X_{(j)}-X_{(n)},
\end{equation*}
for $\frac{n-1}{n}\leq p <1$,  where $1 \leq k \leq n-1$. } 
The ``zero" of $G_n$ was the estimate of the theoretical split point and it was shown that the estimate was consistent, a CLT was proven and its utility was demonstrated on a real-world dataset. 

Our chief interest in this article is to examine an alternative definition for $G$ and consequently $G_n$---we change the domain of the functions from $(0,1)$ to an appropriate subset of the real line. This then permits us to come up with a sample version comprising of truncated sums, at fixed level, as opposed to trimmed sums used in the definition of $G_n$. We first redefine the criterion function. 
\begin{definition}
Consider the extended real line $\mathbb{R}\cup\{-\infty,\infty\}$. Let  $x_+=\sup\{x:F(x)<1\}$ and $x_-=\inf\{x:F(x)>0\}$, the upper and lower endpoints of $F$---they can both be infinite. Then, for $t\in [x_-,x_+]\subseteq \mathbb{R}$, the \emph{cross-over function} is defined as
\begin{equation}\label{T}
T(t)=\frac{1}{F(t)}\displaystyle \int _{-\infty}^t xdF+\frac{1}{1-F(t)}\displaystyle \int _t^{\infty} xdF-2t. 
\end{equation}
\end{definition}
Suppose $p_0$ is the zero of $G$, then the zero of $T$, say $t_0$, would be equal to $Q(p_0)$. Prior to defining the sample-based version of $T$, let us recall the definition of the empirical distribution function. For a sequence $X_i,1 \leq i \leq n$ (not necessarily i.i.d.), consider the usual empirical distribution function 
\[
F_n(t)= \left\{
\begin{array}{l l}
   0 & \quad \text{ if $t < X_{(1)}$}\\
  \frac{k}{n} &\quad \text { if $X_{(k)} \leq t <X_{(k+1)}$}\\
   1       & \quad \text{ if $t \geq X_{(n)}$ }   \\      
\end{array}\right.
\]
where $X_{(1)} \leq, \cdots,\leq X_{(n)}$ a.s. are the order statistics. 
\begin{definition}
For $t \in [X_{(1)},X_{(n)})$ the \emph{sample cross-over function} is defined as

\begin{align}\label{Tn}
T_n(t)&=\frac{1}{F_n(t)}\displaystyle \int _{-\infty}^t xdF_n+\frac{1}{1-F_n(t)}\displaystyle \int _t^{\infty} xdF_n-2t \nonumber\\
&=\frac{1}{nF_n(t)}\bigsum X_i\mathbb{I}_{X_i \leq t}+\frac{1}{n(1-F_n(t))}\bigsum X_i\mathbb{I}_{X_i > t}-2t,
\end{align}
where $\mathbb{I}_B$ is the indicator function of the set $B$. 
\end{definition}
We now extend the definition of $T_n$ onto $(x_-,x_+)$ by defining $a=T_n(X_{(1)})$ and $b=\lim_{t \uparrow  X_{(n)}}T_n$ so that $T_n$ is now defined on $[a,b]$ with 
$-\infty\leq x_-\leq a<b \leq x_+\leq \infty$. The right end-point $b$ exists since $T_n$, by virtue of its definition, is  c\'{a}dl\'{a}g. All results will be proved for $t \in (x_-,x_+)$ since with probability $1$ all limits lie within the compact set $[x_-,x_+]$. Unless stated otherwise, we shall assume throughout that $T_n$ (and its functionals) are all defined on $\tin$.

The problem usually associated with truncated sums is with the decision regarding the level of truncation. In our setup, however, since all levels of truncations are investigated for determining the sample split point, we are untroubled by the issue. The advantage with the truncated sums representation is in the fact that if $X_i$ are observations from a stationary sequence, then $T_n$ is basically a combination of partial sums of the random variables themselves for a fixed truncation level $t$ and not the order statistics. Indeed, literature is rife with limit results for stationary sequences which come in handy while examining the asymptotic behavior of $T_n$. 

\begin{remark}
The function $G$ (and $T$ on its appropriate domain), starting positive, crosses over zero at a point $p_0 \in (0,1)$, the split point. The point of crossing is of chief interest and the conditions on the distribution function $F$ which guarantee the uniqueness of $p_0$ remains an open question (see Remark 1 in \cite{KB2}). However, it is vitally important to control the behavior of $G$ in the vicinity of 0 and 1 since results involving the sample split point for two clusters are proven assuming \emph{only one} crossing of $G$. While it was noted in \cite{KB2} that the condition $\lim \sup _{p \uparrow 1}G(p)<0$ ensured the good behavior of $G$ near 1, the limit results for the sample split point were proved in $[a',b']$ for $0<a'<b'<1$. The function $T$ on the other hand, by virtue of its definition on $[x_-,x_+]$ quite readily satisfies
$$\lim \sup_{t \uparrow x_+}T(t) <0  \quad \text{and} \quad \lim \inf_{t \downarrow x_-}T(t) >0.$$
This property assists us in our definition of the sample split point. 
\end{remark}
Denote by $t_0$ the solution of $T(t)=0$.  We will assume throughout that the solution, $t_0$, is unique. We are now ready to define its empirical counterpart. 
\begin{definition}
For $n \geq 1$, the \emph{sample split point} pertaining to $T_n$ is defined as
\begin{equation*}
 t_n=\left\{
 \begin{array}{l}
 -\infty, \qquad \mbox{ if }T_n (t) <0  \quad \forall \tin;\\
 +\infty, \qquad \mbox{ if } T_n(t) >0 \quad \forall \tin;\\
\sup\{t \in (x_-,x_+): T_n(t) \geq 0\}, \mbox{ otherwise.}
 \end{array}
 \right.
\end{equation*}
\end{definition}

Armed with alternative definitions for the theoretical criterion function, its sample version and the sample split point, we now extend the clustering framework developed in \cite{KB2} to dependent observations arising from a population with a smooth and invertible distribution function. The partial-sums nature of $T_n$ offers a platform to consider almost any kind of dependence structure for which a law of large numbers and a central limit theorem of some sort exist.
\section{Assumptions}\label{assumptions}

\ind Let $\mathbf{X}=\{X_i\}_{i \geq 1}$ be a strictly stationary sequence of real valued random variables with distribution function $F$. In certain cases, if there is no confusion, for ease of notation, we shall denote by $X$ the canonical random variable from $F$. For the sequence $\mathbf{X}$, let
\[
\sigma_m=\sigma(X_1,\cdots,X_m)
\]
be the $\sigma$ field generated by the random variables $X_1,\cdots,X_m$. In similar fashion let
\[
\sigma_{m+k}=\sigma(X_{m+k},X_{m+k+1},\cdots).
\]
 We assume that the random variables satisfy a $\beta$-mixing condition with the $\beta$-mixing coefficient $\beta_k$ for the sequence $\mathbf{X}$ defined as:
\begin{equation}
\beta_k(\mathbf{X})=\sup_{m \geq 1}E \sup \{|P(B|\sigma_m)-P(B)|: B \in \sigma_{m+k}\}.
\end{equation} 
Suppose one obtains a strictly stationary $\beta$-mixing subsequence $X_1,\cdots, X_n$ of $\mathbf{X}$ as observations. We make the following assumptions:
\begin{description}
\item [$A1$] $F$ is invertible for $0<p<1$ and absolutely continuous with density $f$ with respect to the Lebesgue measure.
\item [$A2$] $E(X)=0$ and $E(X^2)=1$.
\item [$A3$] $F$ is twice continuously differentiable at any $t \in (x_-,x_+)$.
\item [$A4$] For $r>1$, $\beta_k=o(k^{-r})$.
\end{description}
Assumption $A4$ does not represent the weakest possible condition on the mixing rate guaranteeing the validity of our results. Our objective in this article is not to refine or develop the limit theory for $\beta$-mixing sequences; we primarily intend to use it as a device to incorporate dependence within observations which are to be clustered. We do not claim to prove our results under the weakest possible conditions. Indeed, to our knowledge, there does not exist a weak, universal condition on the rate of decay of the $\beta$-mixing coefficients which would guarantee convergences in most cases---unless, of course, the strongest known hitherto is assumed. Assumption $A2$ will be clarified in the subsequent section upon noting the invariance of the ``zero" of $T_n$ to scaling and translations of the observations. 
\section{Some remarks on $T$ and $T_n$}\label{remarks}
In this article, as with \cite{KB2}, we only consider the case of two clusters. The criterion function $T$ can be thought of as a statistical functional on the space of distribution functions, which induces a functional on $D[x_-,x_+]$, the space of c\'{a}dl\'{a}g functions on $[x_-,x_+]$. In other words, T can be represented as a statistical functional of $F$ defined as, for $t \in [x_-,x_+]$,
\[
V(F,t)=\frac{1}{F(t)}\int_{x_-}^t(s-t)dF(s)+ \frac{1}{1-F(t)}\int_t^{x_+}(s-t)dF(s).
\]
Equivalently, for fixed $t$, we have
\[
V(F,t)=V_1(F_1,t)+V_2(F_2,t),
\]
where
\begin{equation*}
\begin{aligned}[c]
F_1(s)=\left\{
    \begin{array}{l l}
   \frac{F(s)}{F(t)} & \quad \text{ if $x_- \leq s \leq t$}\\
   1       & \quad \text{ if $t<s<\leq x_+$ }   \\      
   
\end{array}\right.
\end{aligned}
,\qquad 
\begin{aligned}[c]
F_2(s)=\left\{
    \begin{array}{l l}
   \frac{F(s)-F(t)}{1-F(t)} & \quad \text{ if $t<s \leq x_+$}\\
   0       & \quad \text{ if $x_-\leq s\leq t$ }     \\      
  
\end{array}\right.
\end{aligned}
\end{equation*}
with $V_1(F_1,t)=\int(x-t)dF_1(x)$ and $V_2(F_2,t)=\int(x-t)dF_2(x)$. Hence, for fixed $t\in [x_-,x_+]$, $F_1$ and $F_2$ are the distribution functions of a real-valued random variable truncated above and below at $t$ respectively. Then, $V_1$ and $V_2$ are the respective mean functionals with $t$ subtracted and one is perhaps entitled to believe that existing limit results for differentiable statistical functionals can be employed with minimal fuss. Unfortunately, if $t$ is finite and $x_-$ is $-\infty$, for example in $V_1$, then the functional $F \mapsto \int_{[x_-,t]} sdF$ is not Hadamard differentiable with respect to the supremum norm from the domain of all distribution functions to $[x_-, t]$; the same is true for $V_2$ too (see, \cite{van}, problem $7$, p.303). This is so since the supremum norm offers us no control over the tails of the distributions. 

The intuition behind the sample criterion function is simple: $T_n$ is based on the distances between a point $t$ and the means of the observations bounded above and below by $t$. Then, by checking the distances for all possible $t$, one hopes to ascertain the $t$ at which the distances match up, viz., the function $T_n$ becomes $0$. Based on that particular $t$, one is then able to infer if the sample perhaps was obtained from a population with `more than one mean' or estimate the split point.

As in the case of $G_n$ used in \cite{KB2}, we note that the ``zero" of $T_n$ is invariant to scaling and translation of original data. Suppose $\{X_i\}_{1\leq i\leq n}$ is the sequence of interest and if $Z_i=u X_i+v$ for $v \in (x_-,x_+)$ and $u >0$, and we denote by $T^z_n$ the sample criterion function based on $Z_i$ with empirical cdf $F^z_n$, then, 
\begin{align*}
T^z_n(t)=\frac{1}{nF^z_n(t)}\bigsum Z_i\mathcal{I}_{Z_i \leq t}+\frac{1}{n(1-F^z_n(t))}\bigsum Z_i\mathcal{I}_{Z_i>t}-2t.
\end{align*}
It is an easy exercise to check that 
\[
T^z_n(t)=uT_n\Big(\frac{t-v}{u}\Big),
\]
and therefore, the ``zero" of $T_n$ remain unchanged. This elucidates the role played by $F_n$ in the definition of $T_n$: the presence of $F_n$ guarantees the invariance of $t_n$ with respect to scaling and translation.
\section{Preparatory results}\label{prep}
 In this section, in anticipation of the subsequent arguments to be employed in the proofs, we prove a few lemmas and state two results from the literature on empirical processes for dependent sequences. We first state a result from \cite{BY} which will be used in the proof of weak convergence of $T_n$. 
\begin{theorem}{\cite{BY}.}\label{BY}
Suppose $\mathcal{F}$ is a class of functions for which $N(\epsilon, d, \mathcal{F})=O_p(\epsilon^{-w})$ for some $w>0$. If $0 <k\leq 1$, then for any $0<s<k$ and $\epsilon>0$, as $n \to \infty$, 
$$ P\left[\sup_{f \in \mathcal{F}}n^{s/(1+s)}\left|\frac{1}{n}\displaystyle \sum _{i=1}^n f(X_i) -Ef(X_i)\right|>\epsilon\right] \to 0.$$
\end{theorem}
The following theorem by \cite{AY} proves a functional limit theorem for empirical process of $\beta$-mixing sequences indexed by a VC-class (Vapnik-\u{C}ervonenkis class) of functions. 
\begin{theorem}{\cite{AY}.}\label{AY}
Suppose $\mathcal{F}$ is a measurable uniformly bounded VC class of functions. If assumption $A4$ is satisfied, then there is a Gaussian process $\{G(f)\}_{f \in \mathcal{F}}$ which has a version with uniformly bounded and uniformly continuous paths with respect to the $\|.\|_2$ norm such that 
$$\left\{n^{-1/2} \displaystyle \sum _{i=1}^n (f(X_i)-Ef(X_i))\right\}_{f \in \mathcal{F}}\Rightarrow \{G(f)\}_{f \in \mathcal{F}}\qquad \text{ in  } \quad l^{\infty}(\mathcal{F}). $$
\end{theorem}
\begin{lemma}\label{beta_sigma}
If $\{X_i\}_{i\geq1}$ is a stationary sequence satisfying a $\beta$-mixing condition, then for any measurable function $h$, the sequence $\{h(X_i)\}_{i \geq 1}$ satisfies the $\beta$-mixing condition (\cite{BY}). 
\end{lemma}
\begin{proof}
For any measurable function $h$, observe that the $\sigma$-field of $h(X)$ is contained in the $\sigma$-field of $X$. Therefore, the function $h$ has its $\beta$-mixing rate bounded by the corresponding rate of the original sequence and the sequence $\{h(X_i)\}_{i \geq 1}$ satisfies the $\beta$-mixing condition. 
\end{proof}
One way of quantifying the `size' of a class of functions is by the notion of \emph{covering number}; we recall its definition (see p. 275 of \cite{van}, for instance).
\begin{definition}
The covering number $N(\epsilon, \mathcal{F},d)$  related to semi-metric $d$ on a class $\mathcal{F}$ of functions, is the defined as
$$N(\epsilon, \mathcal{F},d)=\text{min}_m \left\{ \text{there are } g_1,\ldots,g_m \text{in  } L^1 \text{ such that } \text{min}_j d(f,g_j) \leq \epsilon \text{ for any $f$ in $\mathcal{F}$}\right \}. $$
\end{definition}
A well-known result from empirical process theory states that  Vapnik-\u{C}ervonenkis (VC) classes are examples of polynomial classes in the sense that their covering numbers are bounded by a polynomial in $\epsilon ^{-1}$ uniformly over all probability measures for which the class $\mathcal{F}$ is not identically zero (see Lemma 19.15, p. 275 of \cite{van}). The following Lemma asserts a Glivenko-Cantelli-type result for the empirical distribution function defined using $\beta$-mixing random variables. 
\begin{lemma}\label{Fn}
For any $\epsilon >0$, under assumptions $A1-A4$, 
\begin{equation*}
P\left[\sup_{t \in [x_-,x_+]}|F_n(t)-F(t)|>\epsilon\right] \to 0 \qquad  \textrm{as } n \to \infty.
\end{equation*}
\end{lemma}
\begin{proof}
Note that the class $\{\mathbb{I}_{(-\infty,t]}: t \in [x_-,x_+]\}$ is a VC class with index bounded above by $(n+1)$ since $F_n(t)$ is based on $n$ random variables. The class hence has a covering number which is $O_p(\epsilon^{-1})$ for $\epsilon >0.$ From Theorem \ref{BY}, we have the result for $\beta$-mixing sequences. Therefore, the classical Glivenko-Cantelli property for sums of random variables is preserved under $\beta$-mixing. 
\end{proof}
Using the result from the preceding Lemma, we quantify the size of the classes of functions which are of interest in the article. 
\begin{lemma}\label{class}
The classes of functions $\mathcal{G}_1=\{X\mathbb{I}_{ X\leq t},  t \in [x_-,x_+]\}$ and $\mathcal{G}_2=\{X \mathbb{I}_{X >t},  t \in [x_-,x_+]\}$ are VC with VC-index $2$.
\end{lemma}
\begin{proof}
Note that the subgraphs of $\mathcal{G}_1$, $\{(t,u):X\mathbb{I}_{ X\leq t}<u \}$ are increasing in $t$. This implies that they cannot shatter any set of $2$ points. More precisely, from Lemma \ref{Fn}, since $\{\mathbb{I}_{(-\infty,t]}: t \in [x_-,x_+]\}$ is a VC class of sets and the subgraphs of $\mathcal{G}_1$ are increasing, the smallest number of points $k$,  for which no set of size $k$ is shattered is 2. Since the VC index of $\mathcal{G}_1$ is finite, it is VC. Furthermore, owing to the fact that $\mathcal{G}_2=\mathcal{G}_1^c$, the VC-property is preserved. 
\end{proof}

\begin{corollary}\label{GC}
For any $\epsilon>0$ and for a fixed $t$, under assumptions, $A1-A4$, 
\begin{enumerate}
\item $P\left[\sup_{t \in [x_-,x_+]}\left|\frac{1}{n}\bigsum X_i\less - EX\lessone\right|>\epsilon\right]\rightarrow 0 \qquad \textrm{as }n \rightarrow \infty.$

\item  $P\left[\sup_{t \in [x_-,x_+]}\left|\frac{1}{n}\bigsum X_i\more - EX\moreone\right|>\epsilon\right]\rightarrow 0 \qquad \textrm{as }n \rightarrow \infty.$
\end{enumerate}
\end{corollary}
\begin{proof}
The proof is an immediate consequence of Lemma \ref{class} and Theorem \ref{BY}.
\end{proof}
It is therefore the case that the classes $\mathcal{G}_1$ and $\mathcal{G}_2$ consisting of truncated-at-fixed level functions have covering numbers which are $O_p(\epsilon^{-1})$ and are therefore VC. Observe that the classes $\mathcal{G}_1$ and $\mathcal{G}_2$ are uniformly bounded by $X$ which is in $L^2$ owing to assumption $A2$.

\section{Asymptotics}\label{asymptotics}
\subsection{Weak convergence of $T_n$}\label{Tnlimits}
We now have all the ingredients required to prove the weak convergence of $T_n$. Suppose we denote $U_n(t)=\sqrt{n}(T_n(t)-T(t))$. The presence of $F_n$ in the definition of $T_n$ renders $U_n(t)$ unsuitable for representation as an empirical process indexed by classes of functions; we are hence unable to use results from empirical processes under $\beta$-mixing to directly obtain the asymptotic normality of $U_n$. For convenience, we define a few quantities first.  Let $\mu_{l}(t)=\frac{1}{F(t)}EX\lessone$ and $\mu_{u}(t)=\frac{1}{1-F(t)}EX\moreone$ . Also, let
\begin{equation}\label{xi}
\xi_t=\frac{1}{F(t)}\Big(X-\mu_{l}(t)\Big)\lessone+\frac{1}{1-F(t)}\Big(X-\mu_{u}(t)\Big)\moreone.
\end{equation}
Note that based on Lemma \ref{beta_sigma}, for fixed $t$, $\{\xi^0_t,\xi^{(i)}_t\}_{i\geq1}$ is a $\beta$-mixing sequence. Denote by $ l^{\infty}[x_-,x_+]$ the space of bounded functions on $[x_-,x_+]$.
\begin{theorem}\label{FCLT}
Under assumptions $A1-A4$, there is a centered Gaussian process $\{U(t),t \in [x_-,x_+]\}$ which has a version with uniformly bounded and uniformly continuous paths with respect to the $\|\cdot\|_2$ norm such that
\[
U_n(t) \Rightarrow U(t) \qquad in \quad l^{\infty}[x_-,x_+],
\]
with non-negative real valued covariance function given by the absolutely convergent series
\begin{equation}\label{sigma}
\sigma_{t}=\displaystyle \sum_{i \in \mathbb{Z}}E(\xi^0_t\xi_t^{(i)}).
\end{equation}
\end{theorem}
\begin{proof}
We follow the technique espoused in \cite{bill} by first showing the convergence of the finite dimensional distributions and then verifying that the sequence $\{U_n\}$ is asymptotically tight and concentrates on a compact set with high probability.  We first prove the finite dimensional convergence by expressing $U_n$ as partial sums of $\beta$-mixing random variables; the result follows then from \cite{Dou}. We then proceed to show that $\{U_n\}_{n \geq 1}$ is tight in an indirect manner by considering $U_n$ component-wise and establishing tightness of the individual components---indeed, then, the sum of tight component sequences would be tight.  \\
\\ 
\emph{Finite dimensional convergence}\\
As mentioned earlier, our intention is to represent $U_n$ as partial sums of $\beta$-mixing random variables plus an error term which goes to zero in probability. 
\begin{align*}
U_n(t)&=\sqrt{n}(T_n(t)-T(t))\\
&= \sqrt{n}\left[\frac{1}{nF_n(t)}\bigsum X_i \less -\frac{1}{F(t)} EX\lessone\right]\\
	&-\sqrt{n}\left[\frac{1}{n(1-F_n(t))}\bigsum X_i \more +\frac{1}{1-F(t)}EX\moreone \right].
\end{align*}
Now, observe that
\begin{align}\label{trim}
 &\quad \sqrt{n}\left[\frac{1}{nF_n(t)}\bigsum X_i \less -\frac{1}{F(t)}EX\lessone \right]\nonumber\\
&=\sqrt{n}\left[\frac{1}{nF_n(t)}\bigsum X_i \less -\frac{1}{F_n(t)}EX\lessone \right]
+\sqrt{n}\left[EX\lessone\left(\frac{1}{F_n(t)}-\frac{1}{F(t)}\right)\right].
\end{align}
Following some cumbersome algebra, the RHS of (\ref{trim}) can be expressed as
\begin{align*}
&\frac{n^{-1/2}}{F_n(t)}\Big(\bigsum \Big[X_i\less -\frac{1}{F(t)}EX\lessone \less \Big]\Big)
=\frac{n^{-1/2}}{F_n(t)} \displaystyle \sum_{i=1}^n\tau_i,
\end{align*}
where
\[
	\tau_i=(X_i -\mu_{l}(t)) \less
\]
are stationary $\beta$-mixing random variables and $E\tau_i=0$ for each $i=1,\cdots,n$. In a similar manner, we have that
\begin{align*}
\sqrt{n}\left[\frac{1}{n(1-F_n(t))}\bigsum X_i \more -\frac{1}{1-F(t)}EX\moreone\right]=\frac{n^{-1/2}}{1-F_n(t)} \displaystyle \sum_{i=1}^n\eta_i,
\end{align*}
where
\[
	\eta_i=(X_i -\mu_{ut}) \more
\]
are stationary $\beta$-mixing random variable with zero mean.
We thus have the convenient representation
\[
	\hspace{-20mm}U_n(t)=n^{-1/2}\left[\frac{1}{F_n(t)}\bigsum\tau_i+\frac{1}{1-F_n(t)}\bigsum\eta_i\right].
\]
At this point, we still do not have the desired representation; we will, if we can `replace' $F_n(t)$ by $F(t)$ in the preceding equation. We now demonstrate that we are, in fact, able to do precisely that. Observe that, owing to assumptions $A1-A4$, 
\begin{equation*}
n^{-1/2}\bigsum \tau_i\left[\frac{1}{F_n(t)}-\frac{1}{F(t)}\right]\overset{P}\rightarrow 0
\end{equation*}
 since, $n^{-1/2}\bigsum \tau_i$ is asymptotically normal from CLT for $\beta$-mixing sequences (see, \cite{Dou}, for instance) and hence bounded in probability and by continuity of the reciprocal, $\frac{1}{F_n(t)}\overset{P}\rightarrow \frac{1}{F(t)}$. Thus,
\[
\qquad \quad n^{-1/2}\left[\frac{1}{F_n(t)}\bigsum \tau_i\right]=n^{-1/2}\left[\frac{1}{F(t)}\bigsum \tau_i\right]+o_p(1).
\]
By a similar argument we can claim that
\[
	n^{-1/2}\left[\frac{1}{1-F_n(t)}\bigsum\eta_i\right]=n^{-1/2}\left[\frac{1}{1-F(t)}\displaystyle \sum_{i=1}^n\eta_i\right]+o_p(1).
\]
Finally, we have our desired representation: for each $\tin$,
\begin{align}\label{sums}
\hspace{15mm}U_n(t)&=n^{-1/2}\left(\displaystyle \sum_{i=1}^n\Big[\frac{1}{F(t)}\tau_i+\frac{1}{1-F(t)}\eta_i\Big]\right)+o_p(1),\nonumber\\
&=n^{-1/2}\displaystyle \sum_{i=1}^n \xi^{(i)}_t+o_p(1) ,
\end{align}
where $\xi_t^{(i)}$ is as defined in (\ref{xi}). Now, $(U_n(t_1),\cdots,U_n(t_k))$ can be represented as sums of $k$-dimensional random vectors $(\xi^{(i)}_{t_1},\ldots,\xi^{(i)}_{t_k})$ for $i=1,\ldots,n$ by virtue of the representation in (\ref{sums}). Using the multivariate version of the CLT for $\beta$-mixing sequences (see \cite{Dou}), $(U_n(t_1),\cdots,U_n(t_k))$ convergences in distribution to a multivariate normal vector.  
\\
\emph{Tightness}\\
\\
Tightness is verified in an indirect manner. Instead of proving directly that $\{U_n\}$ concentrates on a compact set with high probability, we note that $T_n$ comprises of two truncated mean processes in $t$,
\begin{align*}
M^{(1)}_n(t)&=\Big\{\frac{1}{\sqrt{n}}\bigsum (X_i\less - EX\lessone): t \in [x_-,x_+]\Big\}\\
M^{(2)}_n(t)&=\Big\{\frac{1}{\sqrt{n}}\bigsum (X_i\more - EX\moreone): t \in [x_-,x_+]\Big\}.
\end{align*}
Consider the two classes of functions $\mathcal{G}_1$ and $\mathcal{G}_2$ defined in Lemma \ref{class}. Since $\mathcal{G}_1$ and $\mathcal{G}_2$ are both VC class, from Theorem \ref{AY} we have that
$M^{(1)}_n(t)$ and $M^{(2)}_n(t)$  converge weakly to Gaussian processes $M_1(t)$ and $M_2(t)$ in $l^{\infty}(\mathcal{G}_1)$ and  $l^{\infty}(\mathcal{G}_2)$, respectively, with respect to the $\|\cdot\|_2$ norm. This informs us that the two sequences of truncated mean processes are relatively compact. Since $l^{\infty}$ is complete and separable with respect to the $\|\cdot\|_2$ norm and the corresponding metric, using the converse of Prohorov's theorem (see \cite{bill}, p.115) we can claim that each of the truncated mean sequences are tight.  Furthermore, since the sum of tight sequences concentrate on a compact set, we can claim that the process $M^{(1)}_n+M^{(2)}_n$ is tight. The only unresolved issue is the presence of $F_n$ and $(1-F_n)$ in the denominator of $U_n$. Since both the quantities converge uniformly in $(x_-,x_+)$ owing to Lemma \ref{Fn}, we can use Slutsky's lemma to claim that $\frac{M^{(1)}_n}{F_n}$ and $\frac{M^{(2)}_n}{1-F_n}$ are tight and consequently, so is their sum.
\end{proof}
The functional CLT for $T_n$ provides us with a slew of related results which will assist us considerably while considering the asymptotic behavior of $t_n$. 
\begin{corollary}\label{LLN}
Under assumptions $A1-A4$, for any $\epsilon > 0$, we have 
\[
	P\left(\sup_{\tin}|T_n(t)-T(t)|>\epsilon\right)\rightarrow 0 \qquad \textrm{as }n \rightarrow \infty.
\]
\end{corollary}
\begin{proof}
Note that $C[x_-,x_+] \subset D[x_-,x_+] \subset l^{\infty}[x_-,x_+]$. While convergence of $U_n$ to $U$ is in $l^{\infty}[x_-,x_+]$, from Theorem \ref{FCLT} it is the case the that $U \in C[x_-,x_+]$ with probability one (in fact, $U$ has uniformly continuous paths and belongs to a smaller class of functions). Therefore the functional $\sup_{t \in(x_-,x_+)}|U(t)|$ is continuous in $C[x_-,x_+]$ with respect to metric induced by the $\|\cdot\|_2$ norm inherited from $l^{\infty}[x_-,x_+]$. As a consequence, using the mapping theorem (see \cite{DP3} for instance)
\[
\sup_{x_-<t<x_+}\sqrt{n}|T_n(t)-T(t)| \Rightarrow \sup_{x_-<t<x_+}|U(t)|
\]
and it is hence true that $T_n$ converges to $T$ uniformly in $(x_-,x_+)$.
\end{proof}
An important consequence of the functional CLT which will be used in the sequel is stochastic equicontinuity of $U_n$. The proof is straightforward and is omitted. 
\begin{corollary}\label{SE}
Under assumptions $A1-A4$, for every $\epsilon>0$ and $\eta>0$, there exists $\delta>0$ such that 
$$\lim \sup_{n}P\left(\sup_{s,t \in (x_-,x_+):|t-s|<\delta} \left|U_n(t)-U_n(s) \right|>\eta\right)<\epsilon.$$
In other words, $\{U_n(t): t \in (x_-,x_+)\}$ is stochastically equicontinuous. 
\end{corollary}
\subsection{Asymptotics of $t_n$}\label{tnlimits}
At this juncture the hard part has been done and we are in a convenient position of using arguments from \cite{KB2} directly, with minimal changes, in the limit theorems for $t_n$. This is so since the proofs of results for the empirical split point $p_n$ in \cite{KB2} are based solely on the weak convergence and uniform convergence of the sample criterion function; the results in the preceding section play a similar role in that regard and we can therefore repeat arguments from their proofs with very little change.  We are interested in showing consistency of $t_n$ and proving a CLT. Let us enumerate the sequence of steps which will lead us to the required results. 
\begin{enumerate}[(i)]
\item It is first shown that $t_n \overset{P}\to t_0$, where $T(t_0)=0$---i.e, $t_0$ is the true split point. 
\item Under the assumption that $T'(t_0)<0$ (see Theorem 2 in \cite{JH2}), it is then shown that $t_n$ belongs to a $O_p(n^{-1/2})$ neighborhood of $t_0$. 
\item Then, using stochastic equicontinuity of the process $U_n$, it is shown that in a $O(n^{-1/2})$ neighborhood of $t_0$, $T_n$ can be approximated uniformly by a line with slope $T'(t_0)$ and intercept $T_n(t_0)$ upto $o_p(n^{-1/2})$. It is then shown that the $t$ at which the line crosses over zero is within $o_p(n^{-1/2})$ of where $T_n$ crosses over zero, which is $t_n$.
\item Using the preceding result, $t_n$ is shown to have a convenient representation in terms of $t_0$ and $T_n(t_0)$ which immediately leads to the CLT for $t_n$. 
\end{enumerate}   


Usually stochastic equicontinuity of $U_n$ would have been used in conjunction with empirical process techniques to obtain the asymptotics of $t_n$ via an argmax or argmin reasoning. In \cite{KB2} it was noted that their empirical criterion function $G_n$ was not in a form which would have facilitated the representation of the empirical split point $p_n$ as an M-estimator. We face a similar problem here in the sense that $T_n$ cannot be expressed as an estimate of the statistical functional $T$ using the empirical distribution function $F_n$; empirical process techniques, directly, do not aid us as was noted in the arguments for weak convergence. Therefore, we proceed along similar lines as in \cite{KB2} and use a geometric argument to obtain the results for $~t_n$. 

We note now how proofs for asymptotics of the sample split point in \cite{KB2} can be employed with very minimal alterations. The crucial observation lies in the fact that the asymptotics for $t_n$ is \emph{entirely} based on the weak convergence result for $T_n$. At this point, one can quite safely forget the dependence amongst the observations $X_i$ and proceed by minimally altering proofs for $p_n$ in \cite{KB2}. What is pertinent to keep in mind, however, is that $t_n \in [x_-,x_+] \subseteq \mathbb{R}$ whereas in their case the empirical split point $p_n$ was in $[a,b] \subset (0,1)$. The technical arguments, however, remain the same once asymptotics of $T_n$ are provided. 

In \cite{KB2}, stochastic equicontinuity  of $U_n$ defined in terms of trimmed sums was not used in any of the arguments. Instead, the increments of $U_n$ were quantified in their Lemma 2, which is a weaker result in comparison with stochastic equicontinuity. Proof of item (i) follows without change from Theorem 2 in \cite{KB2} with stochastic equicontinuity replacing their usage of Lemma 2.  Proof of item (ii)  is exactly the same as the proof of Lemma 3 in their paper.  We state the results without proof. 
\begin{theorem}\label{consistency}
Assume that $t_0$ is the unique split point. Under assumptions $A1-A4$,  for any $\epsilon>0$,
$$P(|t_n-t_0|>\epsilon) \to 0 \quad \text{as n}\to \infty.$$
\end{theorem}
\begin{lemma}\label{nbd}
Assume $A1-A4$ hold. Suppose $t_0$ is the unique split point in $(x_-,x_+)$ and $T'(t_0)<0$. Then 
$$t_n=t_0+O_p(n^{-1/2}) .$$
\end{lemma}
We shall now prove item (iii).
\begin{lemma}\label{line}
Assume $A1-A4$ hold; suppose $T(t)=0$ has a unique solution $t_0$ and $T'(t_0)<0$, then, uniformly for $t \in  [t_0-\epsilon_n,t_0+\epsilon_n]$,
\[
T_n(t)-T_n(t_0)=T'(t_0)(t-t_0)+o_p(n^{-1/2})
\]
where $\epsilon_n=O(n^{-1/2})$.
\end{lemma}
\begin{proof}
The argument here, using stochastic equicontinuity of $U_n$, is  different from the one in \cite{KB2}. Using the fact that $U_n$ satisfies stochastic equicontinuity, choose $\epsilon>0$ and $\eta>0$  such that
$$ \lim \sup_{n}P\left(\sup_{t \in I_n}\left|U_n(t)-U_n(t_0) \right|>\eta\right)<\epsilon$$
where $I_n=[t_0-\frac{C}{\sqrt{n}},t_0+\frac{C}{\sqrt{n}}]$ for some $C>0$.
Therefore,
\begin{equation}\label{lines}
\sup_{t \in I_n} \sqrt{n}\left|T_n(t)-T_n(t_0)-[T(t)-T(t_0)]\right|=o_p(1)\\
\end{equation}
Now, note that 
\[
T(t)-T(t_0)=T'(t_0)(t-t_0)+O(n^{-1})
\]
owing to assumption $A3$ for all $t \in I_n$. Consequently, substituting in (\ref{lines}), we have the required result. 
\end{proof}
From Lemma \ref{nbd}, although $t_n$ belongs to a $O_p(n^{-1/2})$ of $t_0$, it is not immediately clear how a CLT follows. The following Lemma provides a convenient representation from which asymptotic normality of $t_n$ readily follows. 
\begin{lemma}
Under assumptions $A1-A4$, suppose that $t_0$ is the unique split point and $T'(t_0)<0$. Then as $n \to \infty$,
$$t_n=t_0-\frac{T_n(t_0)}{T'(t_0)} +o_p(n^{-1/2}).$$
\end{lemma}
\begin{proof}
We adopt the argument employed in Lemma 5 of \cite{KB2}. Consider the line given by 
\[
L_n(t)=T_n(t_0)+T'(t_0)(t-t_0).
\]
Denote by $t_n^*$ the random variable which is the solution of 
\[
L_n(t)=0;
\]
that is, almost surely, 
\begin{equation}\label{representation}
t_n^*=t_0-\frac{T_n(t_0)}{T'(t_0)}.
\end{equation}
From the weak convergence of $\sqrt{n}(T_n-T)$ to a Gaussian process, we know that
\[
\sqrt{n}(T_n(t_0)-T(t_0))=\sqrt{n}T_n(t_0)=O_p(1),
\]
and hence $t_n^*=t_0+O_p(n^{-1/2})$. If we can ``replace" $t_n^*$ by $t_n$ in (\ref{representation}), asymptotic normality of $t_n$ would follow automatically. More precisely, it is required that 
\[
\sqrt{n}(t_n-t_n^*)=o_p(1).
\]
From Lemma \ref{nbd}, $t_n=t_0+O_p(n^{-1/2})$ and importantly from Lemma \ref{line}, we have for uniformly for $t$ in $O(n^{-1/2})$ neighborhood of $t_0$
\[
\sqrt{n}(T_n(t)-L_n(t))=\sqrt{n}T_n(t)-\sqrt{n}T_n(t_0)-\sqrt{n}T'(t_0)(t-t_0)=o_p(1).
\]
Since $L_n$ is a line uniformly approximating $T_n$ in $O(n^{-1/2})$ neighborhood of $t_0$, its zero and the ``zero" of $T_n$, $t_n$, are within the same order of deviation. Consequently,  this leads us to claim that $t_n=t_n^*+o_p(n^{-1/2})$ and by substituting $t_n^*$ with $t_n$ is (\ref{representation}), the proof of the Lemma is complete. 
\end{proof}
\begin{theorem}\label{tn_normality}
Assume $A1-A4$ hold and suppose that $t_0$ is the unique split point and $T'(t_0)<0$. Then as $n \to \infty$,
\[
\sqrt{n}(t_n-t_0) \Rightarrow N\left(0,\frac{\sigma_{t_0}}{T'(t_0)^2}\right),
\]
where the non-negative $\sigma_t$ is as defined in Theorem \ref{FCLT}.
\end{theorem}
\section{Numerical example}\label{sim}
Note that the limit theorems presented in the preceding sections are valid for $\phi$-mixing sequences too since $\phi$-mixing represents a stronger condition than $\beta$-mixing, i.e. a $\phi$-mixing sequence is also $\beta$-mixing. It is well-known that for stationary Gaussian sequences $\phi$-mixing is equivalent to $m$-dependence for a positive integer $m$ (see for eg. \cite{bradley}). Bearing in mind assumptions $A1-A3$, we shall consider a two $2$-dependent Gaussian sequence defined as 
\[
X_k=Y_{k-2}+Y_{k-1}+Y_k 
\]
where $Y_k$ is an i.i.d $N(0,1/3)$ sequence. Therefore, $\{X_i\}_{i \geq 1}$ is a stationary sequence from $N(0,1)$ satisfying a $\phi$-mixing condition. Since the density is unimodal and symmetric it is easy to note that the true split point $t_0=0$, $F(t_0)=0.5$, $\mu_l(0)\approx -0.4$ and $\mu_u(0)\approx 0.4$. For the variance $\sigma_0$, we will employ the expression given in \cite{bill} p. 174 for $\phi$-mixing sequences: $\sigma_0=E(\xi^0_0)^2+2\sum_{k=1}^{\infty}E(\xi^0_0\xi^k_0)$. Based on equations (\ref{xi}) and (\ref{sigma}) and keeping in mind that $X_i$ is a 2-dependent sequence, straightforward calculation gives $\sigma_0=3$. Now, note that 
\begin{align*}
T'(t)=\frac{tf(t)}{F(t)}-\frac{f(t)\int_{-\infty}^t xdF}{F^2(t)}
	-\frac{tf(t)}{1-F(t)}+\frac{f(t)\int_{t}^{\infty} xdF}{(1-F(t))^2}-2,
\end{align*}
and hence $T'(0)=4f(0)\left[\int_{-\infty}^0 xdF+\int_{0}^{\infty} xdF\right]-2\approx-0.73$; hence, assumption that $T'(t_0)<0$ in the statement of Theorem \ref{tn_normality} is satisfied. As a consequence, we have that the asymptotic variance of $\sqrt{n}(t_n-t_0)=\sqrt{n}t_n$ is approximately $5.67$. Table 1 below corroborates the theoretical results. 
\begin{table}[th]
\begin{center}
\caption{Simulated mean and variance of $\sqrt{n}t_n$ for different sample sizes for the 2-dependent sequence; $1000$ simulations were performed for each sample size.}
\vspace{2mm}
\begin{tabular}{|c||c|c|c|c|}
\hline
Sample sizes& $n=100$ & $n=300$& $n=1000$& $n=5000$\\ \hline
Simulated Mean &$0.004$& $0.007$&$0.003$ &$-0.001$ \\ \hline
Simulated Variance &$5.63$& $5.76$&$5.71$ &$5.69$ \\ \hline
\end{tabular}
\end{center}
\end{table}

\section{Concluding remarks}
For the i.i.d. case, the results in this article can be obtained without much effort since the framework is a story of truncated-at-fixed-level means. The results can be found in the dissertation of \cite{KB3}. The partial-sums nature of the sample criterion function is an attractive feature when exploring extensions of the present framework to problems in higher dimensions; in contrast, the framework in \cite{KB2} is conceivably harder owing to the usage of order statistics. Some interesting theoretical points arise, however: necessary and sufficient conditions on $F$ guaranteeing uniqueness of the true split point are unknown and are of interest since all limit results are obtain assuming their uniqueness; the presence of $F_n$ and $(1-F_n)$ in the denominator in the expression of $T_n$ bring about some disagreeable technical issues regarding the definition of $T_n$. Their presence is justifiably necessary though since they ensure that the sample split point remains invariant to scaling and translations of the data. 
\section{Acknowledgements}
The author is grateful to Vladimir Pozdnyakov and Dipak Dey for the useful discussions and guidance. He would also like to thank the referees for a detailed review; in particular, he thanks one referee for a meticulous review which improved the overall exposition. 
\vspace{5mm}  

\bibliography{ref}
\bibliographystyle{plainnat}
\end{document}